\documentclass{amsart}
\usepackage{graphicx}
\usepackage{amssymb}
\usepackage{amssymb}
\usepackage{amsmath}
\newtheorem{thm}{Theorem}[section]

\newtheorem{lem}[thm]{Lemma}

\theoremstyle{definition}
\newtheorem{defn}[thm]{Definition}
\theoremstyle{remark}
\newtheorem{rem}[thm]{Remark}
\numberwithin{equation}{section}

\begin{document}

\title[Rate of convergence for a two-scale Galerkin scheme]{Rate of convergence for a Galerkin scheme approximating a two-scale  reaction-diffusion system with nonlinear transmission condition}%
\author{Adrian Muntean}%
\address{Center for Analysis Scientific computing and Applications (CASA), Department of Mathematics and Computer Science, Institute of Complex Molecular Systems (ICMS), Eindhoven University of Technology, PO Box 513, 5600 MB, Eindhoven, The Netherlands}%
\email{a.muntean@tue.nl}%
\author{Omar Lakkis}%
\address{Department of Mathematics, University of Sussex, Falmer, Brighton, BN1 9RF, UK }%
\email{o.lakkis@sussex.ac.uk}%
\subjclass{35 K 57, 65 L 70, 80 A 32, 35 B 27}%

\begin{abstract} We study a two-scale reaction-diffusion system with nonlinear reaction terms and a nonlinear transmission condition (remotely ressembling Henry's law) posed  at air-liquid interfaces.
We prove the rate of convergence of the two-scale Galerkin method proposed in \cite{Maria} for approximating this system in the case when both the microstructure and macroscropic domain are two-dimensional. The main difficulty is created by the presence of a boundary nonlinear term entering the transmission condition. Besides using the particular two-scale structure of the system, the ingredients of the proof include two-scale interpolation-error estimates, an interpolation-trace inequality, and improved regularity estimates.
\end{abstract}
\maketitle

\section{Introduction}

Reaction and transport phenomena in porous media are the governing processes in many
natural and industrial systems. Not only do these reaction and transport phenomena occur
at different space and time scales, but it is also the porous medium itself which is
heterogeneous with heterogeneities present at many spatial scales. The mathematical challenge in this context is to understand and then control the interplay between nonlinear production terms with intrinsic multiple-spatial structure and structured transport in porous media.  To illustrate this scenario, we consider a large domain with randomly distributed heterogeneities where complex
two-phase-two-component processes are relevant only in a small (local) subdomain. This
subdomain (which sometimes is refered to as {\em distributed microstructure}\footnote{Further keywords are: Barenblatt's parallel-flow models
, totally-fissured and partially-fisured media, or double(dual-)-porosity models.} following the terminology of R. E. Showalter) needs fine resolution as the complex processes are governed by small-scale effects. The PDEs used in this particular context need to incorporate two distinct spatial scales: a macroscale (for the large domain, say $\Omega$) and a microscale (for the microstructure, say $Y$). Usually, $x\in \Omega$ and $y\in Y$ denote macro and micro variables.

\subsection{Problem statement}

Let $S$ be the time interval $]0,T[$ for a given fixed $T>0$. We consider the following two-spatial-scale PDE system describing the evolution of the  the vector $(U,u,v)$:

\begin{equation}\label{eqU}
\theta\partial_t U(t,x)- D\Delta U(t,x) = -\int_{\Gamma_R}b(U(t,x)-u(t,x,y))d\lambda^2_y \quad \mbox{ in }  S\times\Omega,
\end{equation}
\begin{eqnarray}
&&\partial_t u(t,x,y)-d_1\Delta_y u(t,x,y)=-k\eta(u(t,x,y),v(t,x,y)) \quad \,\,\, \mbox{ in } S\times\Omega\times Y,\\
&&\partial_t v(t,x,y)-d_2\Delta_y v(t,x,y)=-\alpha k\eta(u(t,x,y),v(t,x,y))\quad \mbox{ in } S\times\Omega\times Y,
\end{eqnarray}\\
with macroscopic non-homogeneous Dirichlet boundary condition
\begin{equation}\label{macro-bc}
U(t,x)=U^{ext}(t,x) \quad \mbox{ on } S\times \partial\Omega,
\end{equation}
and microscopic homogeneous Neumann boundary conditions
\begin{eqnarray}\label{micro-bc}
&&\nabla_y u(t,x,y)\cdot n_y = 0 \quad \mbox{ on } S\times \Omega \times \Gamma_N, \\
&&\nabla_y v(t,x,y)\cdot n_y=0  \quad \mbox{ on } S\times \Omega \times \Gamma.
\end{eqnarray}
The coupling between the micro- and the macro-scale is made by the following nonlinear transmission condition  on $\Gamma_R$
\begin{equation}\label{micro-macro-bc}
-\nabla_y u(t,x,y)\cdot n_y=-b(U(t,x)-u(t,x,y)) \quad \mbox{ on } S \times \Omega\times \Gamma_R.
\end{equation}
The initial conditions
\begin{eqnarray}
U(0,x)& = & U_I(x) \hspace{.7cm} \mbox{ in } \Omega,\\
u(0,x,y)& = & u_I(x,y)  \hspace{.4cm}\mbox{ in } \Omega\times Y,\\
v(0,x,y) & = & v_I(x,y)  \hspace{.44cm} \mbox{ in } \Omega\times Y,\label{eq-ic}
\end{eqnarray}
close the system of mass-balance equations.

Continuing along the lines of \cite{Maria}, the central theme of this paper is understanding the role of the nonlinear term $b(\cdot)$ in what the {\em a priori} and {\em a posteriori} error analyses of (\ref{eqU})--(\ref{eq-ic})  are concerned. Within the frame of this paper, we focus on the {\em a priori} analysis and consequently prepare a functional framework for the {\em a posteriori} analysis which is still missing for such situations. Since our problem is new,  the existing well-established literature on {\em a priori error} estimates for linear two-scale problems (cf. e.g. \cite{HS}) cannot guess the rate of convergence of the Galerkin approximants to the weak solution to (\ref{eqU})--(\ref{eq-ic}). Therefore, a new analysis approach is needed. Notice that the main difficulty is created by the presence of a boundary nonlinear term entering the transmission condition (\ref{micro-macro-bc}). Here we prove the rate of convergence of the two-scale Galerkin method proposed in \cite{Maria} for approximating this system in the case when both the microstructure and macroscropic domain are two-dimensional, see Theorem \ref{rate}. Nevertheless, we expect that the results can be extended to the 3D case under  stronger assumptions, for instance, on the regularity of $\Gamma_R$ and data.  
Besides using the particular two-scale structure of the system, the ingredients of the proof include two-scale interpolation-error estimates, an interpolation-trace inequality, and improved regularity estimates.

The paper is structured in the following fashion:

\tableofcontents

\subsection{Geometry of the domain}\label{geometry}

We assume the domains $\Omega$ and $Y$ to be connected in ${\mathbb{R}}^3$ with Lipschitz continuous  boundaries.  We denote by $\lambda^k$ the $k$-dimensional Lebesgue measure ($k\in \{2,3\}$), and assume that $\lambda^3(\Omega)\neq 0$ and $\lambda^3(Y)\neq 0$. Here, $\Omega$ is the macroscopic domain, while $Y$ denotes the part of a standard pore associated with microstructures within $\Omega$. To be more precise,
$Y$ represents the wet part of the pore.  The boundary of $Y$ is denoted by $\Gamma$, and consists of two distinct parts
$$
\Gamma =\Gamma_R\cup\Gamma_N.
$$
Here $\Gamma_R\cap\Gamma_N = \emptyset$, and $\lambda_y^2(\Gamma_R)\neq 0$. Note that $\Gamma_N$ is the part of $\partial Y$ that is isolated with respect to transfer of mass (i.e. $\Gamma_N$ is a Neumann boundary), while $\Gamma_R$ is the gas/liquid interface
along which  the mass transfer takes place. Throughout the paper $\lambda_y^k$ ($k\in \{1,2\}$) denotes the $k$-dimensional Lebesgue  measure on the boundary $\partial Y$ of the microstructure.

\subsection{Physical interpretation of (\ref{eqU})--(\ref{eq-ic})}

 $U$, $u$, and $v$ are the  mass concentrations assigned to the chemical species $A_1$, $A_2$, and $A_3$ involved in the reaction mechanism
 \begin{equation}\label{reaction}
\mathrm{A_1\leftrightharpoons A_2+A_3 \stackrel{k}{\longrightarrow} H_2O + \mbox{ products}}.
\end{equation}
For instance, the natural carbonation of stone  follows the mechanism (\ref{reaction}), where $\mathrm{A_1:=CO_2(g)}$, $\mathrm{A_2:=CO_2(aq)}$, and  $\mathrm{A_3:=Ca(OH)_2(aq)}$, while the product of reaction is in this case $\mathrm{CaCO_3(aq)}$. We refer the reader to \cite{Aiki} for details on the mathematical analyis of a (macroscopic) reaction-diffusion system with free boundary describing the evolution of (\ref{reaction}) in concrete.

Besides overlooking what happens with the produced $\mathrm{CaCO_3(aq)}$, the PDE system also indicates that we completely neglect the water as reaction product in (\ref{reaction}) as well as its motion inside the microstructure $Y$. A correct modeling of the role of water is possible. However, such an extension of the model would essentially complicate the structure of the PDE system and would bring us away from our initial goal. On the other hand, it is important to observe  that the sink/source term 
\begin{equation}\label{integral_term}
 -\int_{\Gamma_R}b(U-u)d\lambda^2_y
\end{equation}
 models the contribution in the effective equation (\ref{eqU}) coming from mass transfer between air and water regions at microscopic level. Surface integral terms like (\ref{integral_term}) have been obtained in the context of two-scale models (for the so-called Henry and Raoult laws \cite{Danckwerts} -- linear choices of $b(\cdot)$!) by various authors; see for instance \cite{HJM} and references cited therein. 
The parameter $k$ is the reaction constant for the competitive reaction between the species $\mathrm{A_2}$ and $\mathrm{A_3}$, while $\alpha$ is the ratio of the molecular weights of these two species.  Furthermore, we denote by $\theta$ the porosity of the medium.

\section{Technical preliminaries}

\subsection{Assumptions on data, parameters, and spatial domains $\Omega, Y$}

For the transport coefficients, we assume that
\begin{enumerate}
\item[(A1)]  $D>0, d_1>0, d_2>0$.
\end{enumerate}
Concerning the micro-macro transfer and the reaction terms, we suppose:
\begin{enumerate}
\item[(A2)] The sink/source term $b:\mathbb{R}\to \mathbb{R}_+$ is globally Lipschitz, and $b(z)=0$ if $z\leq 0$. This implies that it exists a constant $\hat c>0$ such that $b(z)\leq \hat c z$ if $z>0$;
\item[(A3)] $\eta: \mathbb{R}\times \mathbb{R} \to \mathbb{R}_+$ is defined by $\eta(r,s):= R(r)Q(s)$, where $R,Q$ are globally Lipschitz continuous, with Lipschitz constants $c_R$ and $c_Q$ respectively. Furthermore, we assume that $R(r)>0$ if $r>0$ and $R(r)=0$ if $r \leq 0$,  and similarly, $Q(s)>0$ if $s>0$ and $Q(s)=0$ if $s \leq 0$.

Finally,  we have  $k>0,$ and $ \alpha >0$.
\end{enumerate}
For the initial and boundary functions, we assume
\begin{enumerate}
\item[(A4)] $U^{ext}\in H^1(S,H^2(\Omega)) \cap H^2(S, L^2(\Omega))\cap L^\infty_+(S\times \Omega)$, $U_I \in H^2(\Omega)\cap L^\infty_+(\Omega)$,
$U_I-U^{ext}(0,\cdot)\in H_0^1(\Omega)$,
$u_I,v_I \in L^2(\Omega, H^2(Y))\cap L^\infty_+(\Omega\times Y)$.
\end{enumerate}
For the approximation with piecewise linear functions (finite elements), we assume: 
\begin{enumerate}
\item[(A5)] $\Omega$ and $Y$ are convex domains in $\mathbb{R}^2$ with sufficiently smooth boundaries;
\item[(A6)] $h^2\max\{\gamma_1,\gamma_3\}<1$, where $h,\gamma_1,$ and $\gamma_3$ are strictly positive constants entering the statement of Lemma \ref{interpolation-error}.
\end{enumerate}

\subsection{Weak formulation. Known results}

Our concept of weak solution is given in the following.

\begin{defn}\label{def_weak_formulation}
A triplet of functions $(U, u, v)$ with $(U-U^{ext})\in L^2(S, H_0^1(\Omega))$, $\partial_t U\in L^2(S \times\Omega)$, $(u,v)\in L^2(S, L^2(\Omega, H^1(Y)))^2$, $(\partial_t u, \partial_t v)\in L^2(S \times\Omega\times Y)^2$, is called a weak solution of (\ref{eqU})--(\ref{eq-ic}) if for a.e. $t\in S$ the following identities hold
\begin{eqnarray}
&&\frac{d}{dt}\int_\Omega \theta U\varphi + \int_\Omega D\nabla U\nabla\varphi +\int_\Omega\int_{\Gamma_R} b(U-u)\varphi d\lambda^2_ydx=0\label{weakU}\\
&&\frac{d}{dt}\int_{\Omega\times Y} u\phi + \int_{\Omega\times Y} d_1\nabla_y u\nabla_y\phi - \int_{\Omega}\int_{\Gamma_R} b(U-u)\phi d\lambda^2_ydx \nonumber \\
&&\hspace{5.5cm}+k\int_{\Omega\times Y} \eta(u,v)\phi=0\label{weaku}\\
&&\frac{d}{dt}\int_{\Omega\times Y} v\psi + \int_{\Omega\times Y} d_2\nabla_y v\nabla_y\psi +\alpha k\int_{\Omega\times Y} \eta(u,v)\psi=0,\label{weakv}
\end{eqnarray}
for all $(\varphi,\phi,\psi)\in H_0^1(\Omega)\times L^2(\Omega;H^1(Y))^2$,
and
$$
U(0)= U_I \,\, \mbox{in} \,\, \Omega,\quad u(0)=u_I, \, v(0)=v_I \,\, \mbox{in} \, \,\Omega \times Y.
$$
\end{defn}
\begin{thm} It exists a globally-in-time unique positive and essentially bounded solution $(U, u, v)$ in the sense od Definition \ref{def_weak_formulation}.
\end{thm}
\begin{proof} We refer the reader to \cite{Maria} for the proof of this result.
\end{proof}

\subsection{Galerkin approximation. Basic (semi-discrete) estimates}\label{technical}\label{23}

Following the lines of \cite{Maria,N-RLJ}, we introduce the  Schauder bases: Let $\{\xi_i\}_{i\in\mathbb{N}}$ be a basis of $L^2(\Omega)$, with $\xi_j \in H^1_0(\Omega)$, forming an orthonormal system (say o.n.s.) with respect to $L^2(\Omega)$-norm. Furthermore, let $\{\zeta_{jk}\}_{j,k\in\mathbb{N}}$ be a basis of $L^2(\Omega\times Y)$, with
\begin{equation}\label{zeta}
\zeta_{jk}(x,y)=\xi_j(x)\eta_k(y),
\end{equation}
where $\{\eta_{k}\}_{k\in\mathbb{N}}$ is a basis of $L^2(Y)$, with $\eta_k\in H^1(Y)$, forming an o.n.s. with respect to $L^2(Y)$-norm.

Let us also define the projection operators on finite dimensional subspaces $P^N_x, P^N_y$ associated to the bases $\{\xi_j\}_{ j\in \mathbb{N}}$, and $\{\eta_k, \}_{ k\in \mathbb{N}}$ respectively. For $(\varphi, \psi)$ of the form
\begin{eqnarray}
\varphi(x) &=& \sum_{j\in\mathbb{N}}a_j\xi_j(x), \nonumber\\
\psi(x,y) &= & \sum_{j,k\in\mathbb{N}} b_{jk}\xi_j(x)\eta_k(y), \nonumber
\end{eqnarray}
we define
\begin{eqnarray}
(P^N_x \varphi)(x) &= & \sum_{j=1}^N a_j\xi_j(x),\label{Projxfi} \\
(P^N_x \psi)(x,y) &= & \sum_{j =1}^N \sum_{k\in \mathbb{N}} b_{jk}\,\sigma_j(x)\eta_{k}(y)\label{Projxpsi} \\
(P^N_y \psi)(x,y)&=&\sum_{j \in\mathbb{N}}\sum_{k=1}^N b_{jk}\,\sigma_j(x)\eta_{k}(y).\label{Projypsi}
\end{eqnarray}
The bases $\{\sigma_j\}_{ j\in \mathbb{N}}$, and $\{\eta_k \}_{ k\in \mathbb{N}}$ are chosen such that the projection operators $P^N_{x}, P^N_y$ are stable with respect to the $L^\infty$-norm and $H^2$-norm; i.e. for a given function the $L^\infty$-norm and $H^2$-norm of the truncations by the projection operators can be estimated by the corresponding norms of the function.
\begin{rem}Apparently, this choice of bases is rather restrictive. It is worth noting that we can remove the requirement that $P^N_{x}, P^N_y$ are stable with respect to the $L^\infty$-norm in the case we work with a globally Lipschtz choice for the  mass-transfer term $b(\cdot)$. We will give detailed explanations on this aspect elesewhere.
\end{rem}
Now, we look for finite-dimensional approximations of order $N\in \mathbb{N}$ for the functions $U_0:=U-U^{ext}, u$, and $v $, of the following form
\begin{eqnarray}
U_0^N(t,x) &=& \sum_{j=1}^N\alpha_j^N(t)\xi_j(x), \label{ansatzU}\\
u^N(t,x,y) &= & \sum_{j,k=1}^N \beta_{jk}^N(t)\xi_j(x)\eta_k(y), \label{ansatzu}\\
v^N(t,x,y)&= & \sum_{j,k=1}^N \gamma_{jk}^N(t)\xi_j(x)\eta_k(y),
\end{eqnarray}
where the coefficients $\alpha_j^N, \beta_{jk}^N, \gamma_{jk}^N, j,k = 1, \ldots, N$ are determined by the following relations:
\begin{eqnarray}
&& \int_\Omega \theta \partial_t U^N_0(t) \varphi dx  + \int_\Omega D\nabla U^N_0(t)\nabla\varphi dx = \label{weakUN}\\
&-&  \int_\Omega \left(\int_{\Gamma_R} b\left((U^N_0+ U^{ext}-u^N)(t)\right) d\lambda^2_y + \theta \partial_t U^{ext}(t) + D \Delta U^{ext}(t)\right) \varphi dx \nonumber\\
&&\int_{\Omega\times Y} \partial_t u^N(t)\phi \,dx dy + \int_{\Omega\times Y} d_1\nabla_y u^N(t) \nabla_y\phi \,dx dy =  \label{weakuN}\\
&& \int_{\Omega}\int_{\Gamma_R} b\left((U^N_0+ U^{ext}-u^N)(t)\right)\phi \,d\lambda^2_y dx  - k\int_{\Omega\times Y} \eta \left(u^N(t), v^N(t)\right)\phi \,dy dx\nonumber \\
&&\int_{\Omega\times Y} \partial_t v^N(t) \psi \, dy dx + \int_{\Omega\times Y} d_2\nabla_y v^N(t) \nabla_y\psi \, dy dx = \label{weakvN} \\
&-& \alpha k\int_{\Omega\times Y} \eta \left(u^N(t), v^N(t)\right)\psi\, dy dx\nonumber
\end{eqnarray}
for all
$\varphi \in \mbox{span}\{\xi_j:\ j\in \{1,\dots,N\}\},$ and $\phi, \psi \in \mbox{span}\{\zeta_{jk}:\ j,k\in \{1,\dots,N\}\},$
and
\begin{eqnarray}
\alpha_j^N(0)& := & \int_\Omega (U_I-U^{ext}(0))\xi_j dx, \label{initUN}\\
\beta_{jk}^N(0) & := & \int_\Omega\int_Y u_I \zeta_{jk} dx dy, \label{inituN}\\
\gamma_{jk}^N(0) & :=& \int_\Omega\int_Y v_I \zeta_{jk}dx dy. \label{initvN}
\end{eqnarray}

\begin{thm}\label{basic}
Assume that the projection operators $P^N_{x}, P^N_y$, defined in (\ref{Projxfi})-(\ref{Projypsi}), are stable with respect to the $L^\infty$-norm and $H^2$-norm, and that (A1)--(A4) are satisfied. Then the following statements hold:
\begin{enumerate}
\item[(i)]  The finite-dimensional approximations $U_0^N(t)$, $u^N(t)$, and $v^N(t)$ are positive and uniformly bounded. More precisely, we have for  a.e. $(x,y)\in\Omega\times Y$, all $t\in S$, and all $N\in \mathbb{N}$
\begin{equation}
0 \leq U_0^N(t,x)\leq m_1, \quad 0\leq u^N(t,x,y)\leq m_2, \quad 0\leq v^N(t,x,y)\leq m_3,
\end{equation}
where
\begin{eqnarray}
m_1&:=& 2 ||U^{ext}||_{L^\infty(S\times \Omega)} + ||U_I||_{L^\infty(\Omega)}, \nonumber\\
m_2&:=&\max\{||u_I||_{L^\infty(\Omega\times Y)},m_1\}, \nonumber\\
m_3&:=&||v_I||_{L^\infty(\Omega\times Y)}.\nonumber
\end{eqnarray}
\item[(ii)] There exists a constant $c>0$, independent of $N$, such that
\begin{eqnarray}
||U_0^N||_{L^\infty(S,H^1(\Omega))} + ||\partial_t U_0^N||_{L^2(S,L^2(\Omega))}\leq c,\label{eng1}\\
||u^N||_{L^\infty(S,L^2(\Omega;H^1(Y)))}+||\partial_t u^N||_{L^2(S,L^2(\Omega;L^2(Y)))}\leq c,\\
||v^N||_{L^\infty(S,L^2(\Omega;H^1(Y)))} + ||\partial_t v^N||_{L^2(S,L^2(\Omega;L^2(Y)))}\leq c,\label{eng3}
\end{eqnarray}
\item[(iii)] Then there exists a constant $c>0$, independent of $N$, such that the following estimates hold
\begin{eqnarray}
 || \nabla_x u^N||_{L^\infty(S, L^2(\Omega\times Y)} +  || \nabla_x v^N||_{L^\infty(S, L^2(\Omega\times Y)} &\leq& c \label{estnablax}\\
||\nabla_y \nabla_x u^N||_{L^2(S, L^2(\Omega\times Y)} +  || \nabla_y \nabla_x v^N||_{L^2(S, L^2(\Omega\times Y)} &\leq& c. \label{estnablaxy}
\end{eqnarray}
\end{enumerate}
\end{thm}
\begin{proof}
This statement combines the information stated in Theorem 6.1 and Theorem 6.2 from \cite{Maria}. We refer the reader to the cited paper for the proof details.
\end{proof}
With these estimates in hand, we have enough compactness to establish the convergence of the Galerkin approximates to the weak solution of our problem.
\begin{thm}\label{convergence} There exists a subsequence, again denoted by $(U^N_0, u^N, v^N)$, and a limit $(U_0, u, v) \in L^2(S;H^1(\Omega))\times \left[L^2(S;L^2(\Omega;H^1(Y)))\right]^2$, with $(\partial_t U^N_0, \partial_t u^N, \partial_t v^N) \in L^2(S \times \Omega)\times \left[L^2(S \times \Omega \times Y)\right]^2$, such that
\begin{eqnarray}
&& (U^N_0, u^N, v^N) \rightarrow (U_0, u, v)\,\, \mbox{weakly in} \,\,L^2(S;H^1(\Omega))\times \left[L^2(S;L^2(\Omega;H^1(Y)))\right]^2 \quad\nonumber\\
&& (\partial_t U^N_0, \partial_t u^N, \partial_t v^N) \rightarrow (\partial_tU_0, \partial_t u, \partial_t v)\, \, \mbox{weakly in} \,\, L^2 \nonumber\\
&& (U^N_0, u^N, v^N) \rightarrow (U_0, u, v)\,\, \mbox{strongly in} \,\,L^2 \nonumber\\
&& u^N|_{\Gamma_R} \rightarrow u|_{\Gamma_R} \,\, \mbox{strongly in} \,\,L^2(S\times \Omega, L^2(\Gamma_R))\nonumber
\end{eqnarray}
\end{thm}
\begin{proof}
See the proof of Theorem 6.3 in \cite{Maria}.
\end{proof}
In the next section, we address the question we wish to answer:
\begin{center}
\emph{How fast do the subsequences mentioned in Theorem \ref{basic} converge to their unique limit indicated in Theorem
\ref{convergence}?}
\end{center}
\section{Estimating the rate of convergence: The case $Y\subset \Omega\subset \mathbb{R}^2$}

Adapting some of the working ideas mentioned  in \cite{Thomee,error} to this two-spatial-scale scenario, we obtain an {\em a priori} estimate for the convergence rate of the Galerkin scheme constructed in section \ref{technical}.

\subsection{Approximation of smooth two-scale functions}

As preparation for the definition of the finite element solution to our problem, we briefly introduce some concepts concerning  the approximation of smooth functions in $\Omega, Y\subset \mathbb{R}^2$ (taking into account assumption (A5)); see, for instance, \cite{brenner-scott} or \cite{Thomee} for more details.

For simplicity, we let $h$ denote the maximum length of the sides of the triangulations $\mathcal{T}_h$ of both $\Omega$ and $Y$. $h$ decreases as triangulations are made finer. Let's assume that we can construct quasiuniform triangulations (\cite{Thomee}, p.2) and that the angles of these triangulations are bounded from below by uniformly in $h$ positive constants.

Define $V_h:=\mbox{span}\{\xi_j:\ j\in \{1,\dots,N\}\},$  and $B_h:=\mbox{span}\{\eta_k\ : \ k\in \{1,\dots, N\} \}$ where $\xi_j$ and $\eta_k$ are defined as in section \ref{23}. We also introduce $W_h:=\mbox{span}\{\zeta_{jk}:\ j,k\in \{1,\dots,N\}\}$, where $\zeta_{jk}$ are given by (\ref{zeta}).  Note that $W_h:=V_h\times B_h$. 

A given smooth function $\varphi$ in $\Omega$ vanishing on $\partial \Omega$ may be approximated by the interpolant $I_h\varphi$ in the space of piecewise continuous linear functions  vanishing outside $\bigcup\mathcal{T}_h$. Standard interpolation error arguments ensure that for any $\varphi\in H^2(\Omega)\cap H^1_0(\Omega)$, we get
$$||I_h\varphi -\varphi||_{L^2(\Omega)}\leq c h^2 ||\varphi||_{L^2(\Omega)}$$
$$||\nabla (I_h\varphi-\varphi)||_{L^2(\Omega)}\leq c h ||\varphi||_{L^2(\Omega)}.$$

We define the macro and micro-macro Riesz projection  operators (i.e. $\mathcal{R}_h^M$ and $\mathcal{R}_h^m$) in the following manner: 
\begin{eqnarray}
\mathcal{R}_h^M: H^1(\Omega) \to V_h,\\
\mathcal{R}_h^m: L^2(\Omega;H^1(Y)) \to W_h,
\end{eqnarray}
where $R_h^M$ is the standard single-scale Riesz projection, while  $\mathcal{R}_h^m$ is the tensor product of the projection operators
\begin{eqnarray}
P^{\ell 0}&:&L^2(\Omega)\to V_h\\
P^{\ell 1}&:&H^1(Y)\to B_h. 
\end{eqnarray}

Note that this construction of the micro-macro Riesz projection is quite similar to the one proposed in \cite{HS} (cf. especially the proof of Lemma 3.1 {\em loc. cit.}). The only difference is that we do not require any periodic distribution of the microstructure $Y$. Consequently, if one assumes a periodic covering of $\Omega$ by replicates of $Y$ sets, then one recovers the situation dealt with in \cite{HS}.   

\begin{lem}\label{interpolation-error}(Interpolation-error estimates) Let $\mathcal{R}_h^m$ and $\mathcal{R}_h^M$ be the micro and, respectively, macro Riesz's projection operators.   Then there exist the strictly positive constants $\gamma_\ell$ ($\ell\in\{1,2,3,4\}$), which are independent of $h$, such that the Lagrange intepolants $\mathcal{R}_h^m\phi$ and $\mathcal{R}_h^M\varphi$ satisfy the inequalities:
\begin{eqnarray}
||\varphi-\mathcal{R}_h^M\varphi||_{L^2(\Omega)} &\leq & \gamma_1 h^2||\varphi||_{H^2(\Omega)},\label{i1}\\
||\varphi-\mathcal{R}_h^M\varphi||_{H^1(\Omega)}& \leq & \gamma_2 h||\varphi||_{H^2(\Omega)}\label{i2},\\
||\varphi-\mathcal{R}_h^m\phi||_{L^2(\Omega;L^2(Y))}& \leq & \gamma_3 h^2\left(||\phi||_{L^2(\Omega;H^2(Y))\cap L^2(Y;H^2(\Omega))}\right),\label{i3}\\
||\phi-\mathcal{R}_h^m\phi||_{L^2(\Omega;H^1(Y))}&\leq & \gamma_4 h\left(||\phi||_{L^2(\Omega;H^2(Y))\cap L^2(Y;H^2(\Omega))}\right)\label{i4}
\end{eqnarray}
for all $(\varphi,\phi)\in H^2(\Omega)\times \left[L^2(\Omega;H^2(Y))\cap L^2(Y;H^2(\Omega)\right].$
\end{lem}
\begin{proof} (\ref{i1}) and (\ref{i2}) are standard interpolation-error estimates, see \cite{Thomee}, e.g., while  (\ref{i3}) and (\ref{i4}) are interpolation-error estimates especially tailored for elliptic problems with two-spatial scales structures; see Lemma 3.1  \cite{HS} (and its proof) for a statement refering to the periodic case with $(n-1)$-spatially separated scales. One of the key ideas of the proof is to see the spaces $L^2(\Omega,L^2(Y))$ and $L^2(\Omega, H^1(Y))$ as tensor products of the spaces $L^2(\Omega)$ and $L^2(Y)$, and respectively of $L^2(\Omega)$ and $H^1(Y)$. 
\end{proof}

\begin{rem}
Note that, without essential differences, this study can be done in terms of two distinct triangulations  $\mathcal{T}_{h_M}$ and $\mathcal{T}_{h_m}$, where $h_M$ and $h_m$ are maximum length of the sides of the corresponding  triangulation of the macro and micro domains ($\Omega$ and $Y$).
\end{rem}

Unless otherwise specified, the expressions $|\cdot|$ and $||\cdot||$ denote the $L^2$ and $H^1$ norms, respectively, in the corresponding function spaces.
\newpage

\subsection{Main result. Proof of Theorem \ref{rate}}

\begin{defn}\label{def_disc} (Weak solution of semi-discrete formulation) The triplet $(U_0^h,u^h,v^h)$ is called weak solution of the semi-discrete formulation (\ref{weakuN})-(\ref{weakvN}) if  and only if
\begin{eqnarray}
&& \int_\Omega \theta \partial_t U^h_0(t) \varphi dx  + \int_\Omega D\nabla U^h_0(t)\nabla\varphi dx =\\
&-&  \int_\Omega \left(\int_{\Gamma_R} b\left((U^h_0+ U^{ext}-u^h)(t)\right) d\lambda^1_y + \theta \partial_t U^{ext}(t) + D \Delta U^{ext}(t)\right) \varphi dx \nonumber\\
&&\int_{\Omega\times Y} \partial_t u^h(t)\phi \,dx dy + \int_{\Omega\times Y} d_1\nabla_y u^N(t) \nabla_y\phi \,dx dy =  \label{weakuh}\\
&& \int_{\Omega}\int_{\Gamma_R} b\left((U^h_0+ U^{ext}-u^h)(t)\right)\phi \,d\lambda^1_y dx  - k\int_{\Omega\times Y} \eta \left(u^h(t), v^h(t)\right)\phi \,dy dx\nonumber \\
&&\int_{\Omega\times Y} \partial_t v^h(t) \psi \, dy dx + \int_{\Omega\times Y} d_2\nabla_y v^h(t) \nabla_y\psi \, dy dx = \label{weakvh} \\
&-& \alpha k\int_{\Omega\times Y} \eta \left(u^h(t), v^h(t)\right)\psi\, dy dx\nonumber
\end{eqnarray}
for all
$\varphi \in V_h$ and $(\phi, \psi)\in W_h\times W_h$ and $U_0^h(0)=U_I\in L^2(\Omega)$ and $u^h(0),v^h(0)\in L^2(\Omega\times Y)$.
\end{defn}

\begin{lem}(Improved regularity)\label{improved-regularity} Assume (A1)--(A5) to hold. Then
\begin{eqnarray}
&&U_0^h\in L^2(S;H^2(\Omega))\\
&&u^h,v^h\in L^2(S;L^2(\Omega;H^2(Y)))\cap L^2(S;L^2(Y;H^2(\Omega))).
\end{eqnarray}
\end{lem}
\begin{proof}
Assumption (A5) and a standard lifting regularity argument leads to $U_0^h\in L^2(S;H^2(\Omega))$ and $u^h,v^h\in L^2(S\times \Omega;H^2(Y)))$. 	Employing difference quotients with respect to the variable $x$ (quite similarly to the proof of Theorem 6.2 \cite{Maria}), we can show that  $u^h,v^h\in L^2(S\times Y;H^2(\Omega)))$. We omit the proof details.
\end{proof}
\begin{thm}\label{rate}(Rate of convergence)
Assume (A1)--(A5) are satisfied. If additionally, assumption (A6) holds, then it exists a constant $\mathcal{K}>0$, which is independent of $h$, such that
\begin{eqnarray}
||U_0-U_0^h||_{L^2(S;H^1(\Omega)}^2&+&||u-u^h||_{L^2(S;L^2(\Omega;H^2(Y)))\cap L^2(S;L^2(Y;H^2(\Omega)))}^2\nonumber\\
&+&||v-v^h||_{L^2(S;L^2(\Omega;H^2(Y)))\cap L^2(S;L^2(Y;H^2(\Omega)))}^2\leq \mathcal{K} h^2.
\end{eqnarray}
\end{thm}
\begin{rem} We will compute the constant $\mathcal{K}$ explicitly; see (\ref{K}).
\end{rem}
\begin{proof} (of Theorem \ref{rate})
Firstly, we denote the errors terms by
\begin{eqnarray}
e_U&:=& U_0-U_0^h\nonumber\\
e_u&:=& u-u^h\nonumber\\
e_v&:=& v-v^h.\nonumber
\end{eqnarray}
We  choose as test functions in Definition \ref{def_disc} the triplet
\begin{equation}
(\varphi,\phi,\psi):=(r^h-U_0^h,p^h-u^h,q^h-v^h),
\end{equation}
where the functions $r^h$, $p^h$, and $q^h$ will be chosen in a precise way (in terms of Riesz projections of the unknowns) at a later stage.
We obtain
\begin{eqnarray}
\frac{\theta}{2}\frac{d}{dt}|U_0-U^h|^2&+&D||U-U^h||^2\leq \int_\Omega \theta \partial_t(U_0-U^h)(U_0-U^h)\nonumber\\
&+ &\int_\Omega D\nabla(U_0-U_0^h)\nabla(U_0-U^h)\nonumber\\
&= &\theta\int_\Omega  \partial_t (U_0-U^h)(U_0-r^h)+\int_\Omega D\nabla (U_0-U^h)\nabla (U_0-r^h)\nonumber\\
&+&\theta\int_\Omega  \partial_t (U_0-U^h)(r^h-U^h)+\int_\Omega D\nabla (U_0-U^h)\nabla (r^h-U^h).
\end{eqnarray}
Using Cauchy-Schwarz inequality, we have
\begin{eqnarray}\label{exp}
\frac{\theta}{2}\frac{d}{dt}|U_0-U^h|^2&+&D||U_0-U^h||^2\leq \theta |\partial_t (U_0-U^h)||U-r^h|\nonumber\\
&+&D|\nabla(U_0-U^h)||\nabla(U_0-r^h)|\nonumber\\
&+&\theta |\partial_t (U_0-U^h)||r^h-U^h|+D|\nabla(U_0-U^h)||\nabla(r^h-U^h)|\nonumber\\
&\leq & \theta |\partial_t (U_0-U^h)||U-r^h|+D|\nabla(U_0-U^h)||\nabla(U_0-r^h)| \nonumber\\
&+& \int_\Omega \int_{\Gamma_R}|b(U_0-u)-b(U_0^h-u^h)||r^h-U^h|d \lambda_y^1.
\end{eqnarray}
Noticing that $r^h-U_0=(r^h-U_0)+(U_0-U^h)$, (\ref{exp}) leads to
\begin{eqnarray}\label{alpha}
\frac{\theta}{2}\frac{d}{dt}|e_U|^2+D||e_U||^2&\leq & \theta |\partial_t e_U||U-r^h|+ D|\nabla e_U||\nabla (U_0-r^h)|\nonumber\\
&+& \hat c\int_\Omega\int_{\Gamma_R}\left(|e_U|+|e_u|\right)\left(|r^h-U_0|+|e_U|\right)d\lambda_y^1.
\end{eqnarray}
Proceeding similarly with the remaining two equations, we get:
\begin{eqnarray}\label{beta}
\frac{1}{2}|\partial_t e_u|^2&+&d_1 |\nabla_y e_u |^2 \leq |\partial_t (u-u^h)||u-p^h|+d_1|\nabla(u-u^h)||\nabla(u-p^h)|\nonumber\\
&+& |\partial_t (u-u^h)||p^h-u^h|+d_1|\nabla(u-u^h)||\nabla(p^h-u^h)|\nonumber\\
&\leq & |\partial_t e_u||u-p^h|+d_1|\nabla_y e_u||\nabla_y(u-p^h)|\nonumber\\
&+& \int_\Omega \int_{\Gamma_R}|b(U_0-u)-b(U^h-u^h)||p^h-u^h|d\lambda_y^1\nonumber\\
&+&k\int_{\Omega\times Y}|\eta(u,v)-\eta(u^h,v^h)||p^h-u^h|\nonumber\\
&\leq & |\partial_t e_u||u-p^h|+d_1|\nabla_y e_u ||\nabla(u-p^h)|\nonumber\\
&+ &\hat c\int_\Omega \int_{\Gamma_R}\left(|e_U|+|e_u|\right)\left(|p^h-u|+|e_u|\right)d\lambda_y^1\nonumber\\
&+& k\int_{\Omega\times Y}|R(u)Q(v)-R(u^h)Q(v^h)|\left(|p^h-u|+|e_u|\right).\nonumber\\
\end{eqnarray}
Finally, we also obtain
\begin{eqnarray}\label{gamma}
\int_{\Omega\times Y}|\partial_t e_v|^2 &+&d_2\int_{\Omega\times Y}|\nabla_y e_v|^2 \leq |\partial_t e_v||v-q^h|+d_2|\nabla_y e||\nabla_y (v-q^h)|\nonumber\\
&+&\alpha k\int_{\Omega\times Y}|R(u)Q(v)-R(u^h)Q(v^h)| \left(|q^h-v|+|e_v|\right).
\end{eqnarray}
Putting together (\ref{alpha}), (\ref{beta}), and (\ref{gamma}), we obtain
\begin{eqnarray}
\frac{\theta}{2}\frac{d}{dt}|e_U|^2&+&\frac{1}{2}\frac{d}{dt}|e_u|^2+\frac{1}{2}\frac{d}{dt}|e_v|^2+D||e_U||^2\nonumber\\
&+&d_1||e_u||^2+d_2||e_v||^2\leq \theta |\partial_t e_U||U_0-r^h|\nonumber\\
&+&|\partial_t e_u||u-p^h|+|\partial_t e_v||v-q^h|+D||e_U|||\nabla(U_0-r^h)|\nonumber\\
&+&d_1||e_v|||\nabla(v-p^h)|+d_2||e_v|||\nabla_y(v-q^h)|\nonumber\\
&+&\hat c\int_\Omega \int_{\Gamma_R}\left(|e_U|+|e_u|\right)\left(|r^h-U_0|+|e_U|\right)d\lambda_y^1\nonumber\\
&+&\hat c\int_\Omega\int_{\Gamma_R}\left(|e_U|+|e_u|\right)\left(|p^h-u|+|e_u|\right)d\lambda_y^1\nonumber\\
&+&\int_{\Omega\times Y}k(1+\alpha)|R(u)Q(v)-R(u^h)Q(v^h)|\left(|p^h-u|+|q^h-v|+|e_u|+|e_v|\right) \nonumber\\
&=:&\sum_{\ell=1}^4 I_\ell,\nonumber
\end{eqnarray}
where the terms $I_\ell$ ($\ell\in \{1,\dots,4\}$) are given by
\begin{eqnarray}
I_1 &:= & \theta |\partial_t e_U||U_0-r^h|+|\partial_t e_u||u-p^h|+|\partial_t e_v||v-q^h|\nonumber\\
I_2 &:=& D||\nabla e_U|||\nabla (U_0-r^h)|+d_1|\nabla_y e_u||\nabla_y(u-p^h)|+d_2|\nabla_y e_v||\nabla_y(v-q^h)|\nonumber\\
I_3&:=& \hat c\int_\Omega \int_{\Gamma_R}\left(|e_U|+|e_u|\right)\left(|r^h-U_0|+|p^h-u|+|e_U|+|e_u|\right)d\lambda_y^1\nonumber\\
I_4 &:=& k(1+\alpha)\int_{\Omega\times Y}|R(u)Q(v)-R(u^h)Q(v^h)|\left(|p^h-u|+|q^h-v|+|e_u|+|e_v|\right).\nonumber
\end{eqnarray}
We choose now $r^h,p^h$, and $q^h$ to be the respective Riesz projections of $U_0^h$, $u^h,$ and $v^h$ and estimate each of these $I_\ell$ terms, i.e. we set
\begin{equation}r^h:=\mathcal{R}_h^MU^h,\ p^h:=\mathcal{R}_h^m u^h, \mbox{ and } q^h:=\mathcal{R}_h^m v^h.
\end{equation}
The main ingredients used in getting the next estimates are Young's inequality, the interpolation-error estimates stated in Lemma \ref{interpolation-error}, the improved regularity estimates from Lemma \ref{improved-regularity}, as well as an interpolation-trace inequality (see the appendix in \cite{GP}, e.g.).

Let us denote for terseness $$X:=L^2(S;L^2(\Omega;H^2(Y)))\cap L^2(S;L^2(Y;H^2(\Omega))).$$

We obtain following estimates:

\begin{eqnarray}
|I_1| &\leq & \gamma_1\theta|\partial_te_U|h^2||U_0||_{H^2(\Omega)} + \gamma_3 h^2\left(|\partial_te_u + \partial_te_v \right)\left(||u||_{X}+||v||_X\right)\nonumber\\
&\leq & h^2\frac{\gamma_1\theta}{2}\left(|\partial_t e_U|^2+||U_0||^2_{H^2(\Omega)}\right)+h^2\frac{\gamma_3}{2}\left(|\partial_t e_u|^2+||u||^2_X\right)\nonumber\\
&+& h^2\frac{\gamma_3}{2}\left(|\partial_t e_v|^2+||v||^2_X\right).
\end{eqnarray}

\begin{eqnarray}
|I_2|&\leq& \gamma_2 D||\nabla e_U||h||U||_{H^2(\Omega)}+\gamma_4d_1|\nabla_y e_u|h||u||_{X}+\gamma_4 d_2|\nabla_y e_v| h||v||_X\nonumber\\
&\leq & \epsilon |\nabla e_U|^2+h^2c_\epsilon\gamma_2^2D^2||U_0||^2_{H^2(\Omega)}+\epsilon |\nabla_y e_u|^2+h^2 c_\epsilon \gamma_4^2d_1^2||u||^2_X\nonumber\\
&+& \epsilon |\nabla_y e_v|^2+h^2 c_\epsilon \gamma_4^2d_2^2||v||^2_X\nonumber\\
&\leq & h^2 c^* c_\epsilon \left(\gamma_2^2+2\gamma_4^2\right)\left(D^2+d_1^2+d_2^2\right)\left(||U_0||_{H^2(\Omega)}+||u||_X^2+||v||_X^2\right)\nonumber\\
&+& \epsilon|\nabla e_U|^2+\epsilon|\nabla_ye_u|^2+\epsilon|\nabla_ye_v|^2,
\end{eqnarray}
where the constant $c^*>0$ is sufficiently large.

The estimate on $|I_3|$ is a bit delicate. To get it, we repeatedly use the following interpolation-trace estimate
\begin{equation}
||\varphi||^2_{L^2(\Omega); L^2(\Gamma_R)}\leq \epsilon \int_\Omega|\nabla_y \varphi|_{L^2(Y)}^2+c(c_\epsilon +1)||\varphi||^2_{L^2(\Omega;L^2(Y))},
\end{equation}
for the case when $\varphi\in \{e_u,e_v\}$,
where $\epsilon> 0$ and $c, c_\epsilon \in ]0,\infty[$ are fixed constants. We get
\begin{eqnarray}
|I_3| &\leq & \hat c \lambda(\Gamma_R)\int_\Omega |e_U||r^h-U_0|+\hat c \int_\Omega|r^h-U_0|\int_{\Gamma_R}|e_u|d\lambda_y^1\nonumber\\
&+& \int_\Omega |e_U|\int_{\Gamma_R}|p^h-u|+\hat c\int_\Omega\int_{\Gamma_R}|e_u||p^h-u|d\lambda_y^1 +2\hat c\int_\Omega\int_{\Gamma_R} \left(|e_u|^2+|e_v|^2\right)d\lambda_y^1\nonumber\\
&\leq & \frac{\hat c\lambda(\Gamma_R)}{2}\left(||e_U||^2_{H^2(\Omega)}+\gamma_1h^4||U_0||^2_{H^2(\Omega)}\right)\nonumber\\
&+& \frac{\hat c}{2}\left(\gamma_1 \lambda(\Gamma_R) h^4||U_0||^2_{H^2(\Omega)}+||e_u||_{L^2(\Omega;L^2(\Gamma_R))}\right)\nonumber\\
&+& \frac{\hat c}{2}\left(|\lambda(\Gamma_R)||e_U||^2_{H^2(\Omega)}+\epsilon h^2\gamma_4^2||u||_X^2+c(c_\epsilon+1)\gamma_3h^4||u||^2_X\right)\nonumber\\
&+& \frac{\hat c}{2}\left(\epsilon \int_\Omega |\nabla_y e_u|^2+c(c_\epsilon+1)||e_u||^2_{L^2(\Omega;L^2(Y))}+\epsilon h^2 ||u||^2_X +c(c_\epsilon+1)\gamma_3 h^4 ||u||_X^2\right)\nonumber\\
&+& 2\hat c\lambda(\Gamma_R)|e_U|^2+\epsilon\int_\Omega |\nabla_y e_u|^2+c(c_\epsilon+1)||e_u||^2_{L^2(\Omega ;L^2(Y))}.\nonumber\\
\end{eqnarray}

 In order to estimate from above the term $|I_4|$, we use the structural assumption (A3) on the reaction terms $R(\cdot)$ and $Q(\cdot)$. We obtain
\begin{eqnarray}
|I_4| &\leq & k(1+\alpha)\int_{\Omega\times Y}|R(u)-R(u^h)||Q(v)|+|Q(v)-Q(v^h)||R(u^h)|\times\nonumber\\
&\times&\left(|p^h-u|+|q^h-v|+||e_u|+|e_v|\right)\nonumber\\
&\leq & 3 h^2 k \gamma_3(1+\alpha)(Q_mc_R +R_m c_Q)\left(||u||^2_X+||v||^2_X\right)\nonumber\\
&+&k(1+\alpha)(Q_m c_R+R_mc_Q)\left(|e_u|^2+|e_v|^2\right),
\end{eqnarray}
where $R_m:=\max_{r\in [0,M_2]} \{R(r)\},\ Q_m:=\max_{s\in [0,M_3]} \{R(s)\}$, while $c_R$ and $c_Q$ are the corresponding Lipschitz constants of $R$ and $Q$.

Consequently, we obtain
\begin{eqnarray}\label{sum}
\sum_{\ell=1}^3 |I_\ell|&\leq & h^2\left(\frac{\gamma_1}{2}\theta |\partial_t e_U|^2+\frac{\gamma_3}{2}|\partial_t e_u|^2+|\partial_t e_v|^2\right)+h^2(\mathcal{K}+\mathcal{F}(h))\nonumber\\
& + & \left[k(1+\alpha)(Q_mc_R+R_mc_Q)+\left(c+\frac{\hat c}{2}\right)\right]\left(|e_u|^2+|e_v|^2\right)\nonumber\\
&+& \epsilon|\nabla e_U|^2 + \epsilon\left(2+\frac{\hat c}{2}\right)\int_\Omega|\nabla_y e_u|^2+\epsilon \int_\Omega|\nabla_y e_v|^2,
\end{eqnarray} 
where
\begin{eqnarray}
\mathcal{K}&:= & \frac{\gamma_3}{2}\left(||u||^2_X+||v||^2_X\right)\nonumber\\
&+& \frac{\gamma_1\theta}{2}||U_0||^2_{H^2(\Omega)}+3k\gamma_3(1+\alpha(Q_m c_R+R_m c_Q)\left(||u||_X^2+||v||_X^2\right)\label{K}\\
\frac{\mathcal{F}(h)}{h^2}&:=& 2\epsilon \gamma_4||u||^2_X+\gamma_1\left(1+\frac{\hat c}{2}\lambda(\Gamma_R)\right)||U_0||^2_{H^2(\Omega)}\nonumber\\
&+& 2c(c_\epsilon+1)\gamma_3||u||^2_X.
\end{eqnarray}
Notice that $\mathcal{K}$ is a finite positive constant that is independent of $h$, while   $F:]0,\infty[\to ]0,\infty[$ is a  function of order of $\mathcal{O}(h^2)$ as $h\to 0$.

By (A6) we can compensate the first term of the r.h.s. of (\ref{sum}), while the last three terms from the r.h.s. can be compensated by choosing the value of $\epsilon$ as $\epsilon\in \left]0,\min\{D, d_2,\frac{2d_1}{\hat c+4}\}\right[$. Relying on the way we approximate the initial data, we  use now Gronwall's inequality to conclude the proof of this theorem.
\end{proof}

\section*{Acknowledgments}
We thank Maria-Neuss Radu (Heidelberg) for fruitful discussions on the analysis of two-scale models. Partial financial support from British Council Partnership Programme in Science  (project number PPS RV22) is gratefully acknowledged. 

\bibliographystyle{amsplain}

\bibliographystyle{plain}

\end{document}